\documentclass{amsart}
\usepackage[english]{babel}
\usepackage{enumitem}
\usepackage[foot]{amsaddr}
\usepackage[normalem]{ulem}
\usepackage{amsmath, amssymb, amsthm, amscd,color,comment}
\usepackage{amsfonts,mathrsfs}
\usepackage{amssymb}
\theoremstyle{plain}
\newtheorem{theorem}{Theorem}[section]
\newtheorem{lemma}[theorem]{Lemma}
\newtheorem{remark}[theorem]{Remark}

\newtheorem{corollary}[theorem]{Corollary}
\newtheorem{proposition}[theorem]{Proposition}
\theoremstyle{definition}
\newtheorem{question}{Question}
\newcommand{\cC}{\mathcal{C}}
\newcommand{\cG}{\mathcal{G}}
\newcommand{\K}{\mathbb{K}}
\newcommand{\F}{\mathbb {F}}
\newcommand{\PG}{\mathrm{PG}}
\newcommand{\abs}[1]{\lvert#1\rvert}

\def\zhou#1 {\fbox {\footnote {\ }}\ \footnotetext { From Yue: {\color{red}#1}}}
\def\daniele#1 {\fbox {\footnote {\ }}\ \footnotetext { From Daniele: {\color{blue}#1}}}

\title{Asymptotics of Moore exponent sets}
\author{Daniele Bartoli\textsuperscript{\,1}}
\author{Yue Zhou\textsuperscript{\,2\,$\dagger$}}
\address{\textsuperscript{1}Department of Mathematics and Computer Science, University of Perugia, 06123 Perugia, Italy}
\email{daniele.bartoli@unipg.it}
\address{\textsuperscript{2}Department of Mathematics, National University of Defense Technology, 410073 Changsha, China}
\address{\textsuperscript{$\dagger$}Corresponding Author}
\email{yue.zhou.ovgu@gmail.com}
\date{\today}

\begin{document}
\begin{abstract}
	Let $n$ be a positive integer and $I$ a $k$-subset of integers in $[0,n-1]$. Given a $k$-tuple $A=(\alpha_0, \cdots, \alpha_{k-1})\in \F^k_{q^n}$, let $M_{A,I}$ denote the matrix $(\alpha_i^{q^j})$ with $0\leq i\leq k-1$ and $j\in I$. When $I=\{0,1,\cdots, k-1\}$, $M_{A,I}$ is called a Moore matrix which was introduced by E.\ H.\ Moore in 1896. It is well known that the determinant of a Moore matrix equals $0$ if and only if $\alpha_0,\cdots, \alpha_{k-1}$ are $\F_q$-linearly dependent. We call $I$ that satisfies this property a Moore exponent set. In fact, Moore exponent sets are equivalent to maximum rank-distance (MRD) code with maximum left and right idealisers over finite fields. It is already known that $I=\{0,\cdots, k-1\}$ is not the unique Moore exponent set, for instance, (generalized) Delsarte-Gabidulin codes and the MRD codes recently discovered in \cite{csajbok_MRD_maximum_idealisers_arxiv} both give rise to new Moore exponent sets. By using algebraic geometry approach, we obtain an asymptotic classification result: for $q>5$, if $I$ is not an arithmetic progression, then there exists an integer $N$ depending on $I$ such that $I$ is not a Moore exponent set provided that $n>N$.
\end{abstract}
\maketitle

\section{Introduction}

Let $q$ be a prime power and $n$ a positive integer. For a given $k$-tuple
$A:=(\alpha_0,\alpha_1,\ldots,\alpha_{k-1}) \in \F^k_{q^{n}}$, $k\leq n$,  a \emph{square Moore matrix} is defined as
\begin{equation*}
	M_{A}:=\left(
	\begin{matrix}
	\alpha_0 & \alpha_0^{q} & \cdots & \alpha_0^{q^{k-1}} \\ 
	\alpha_1 & \alpha_1^{q} & \cdots & \alpha_1^{q^{k-1}} \\ 
	\vdots&  \vdots & \ddots &  \vdots \\ 
	\alpha_{k-1} & \alpha_{k-1}^{q} & \cdots & \alpha_{k-1}^{q^{k-1}}
	\end{matrix} 
	\right),
\end{equation*}
which is a $q$-analog of the Vandermonde matrix introduced by Moore \cite{moore_two-fold_1896}. The determinant of $M_A$ can be expressed as
\begin{equation*}
\det(M_A)=\prod_{\mathbf{c}} (c_0\alpha_0 + c_1\alpha_1 +\cdots c_{k-1}\alpha_{k-1}),
\end{equation*}
where $\mathbf{c}=(c_0,c_1,\cdots,c_{k-1})$ runs over all direction vectors in $\F_q^k$, or equivalently we can say that $\mathbf{c}$ runs over $\PG(k-1,q)$. In other words, 
\begin{equation}\label{eq:property_Moore}
	\det(M_A)=0 \text{ if and only if }\alpha_0,\cdots,\alpha_{k-1} \text{ are }\F_q\text{-linearly dependent}.
\end{equation}
We call $\det(M_A)$ the \emph{Moore determinant}. 

We may replace the exponents of those elements in $M_A$ in the following way:
For $I=\{ i_0, i_1, \cdots, i_{k-1} \}\subseteq \mathbb{Z}_{\geq 0}$ and $A=(\alpha_0,\alpha_1,\ldots,\alpha_{k-1}) \in \F^k_{q^{n}}$, define
\begin{equation*}
M_{A,I}:=\left(
\begin{matrix}
\alpha_0^{q^{i_0}} & \alpha_0^{q^{i_1}} & \cdots & \alpha_0^{q^{i_{k-1}}} \\ 
\alpha_1^{q^{i_0}} & \alpha_1^{q^{i_1}} & \cdots & \alpha_1^{q^{i_{k-1}}} \\ 
\vdots&  \vdots & \ddots &  \vdots \\ 
\alpha_{k-1}^{q^{i_0}} & \alpha_{k-1}^{q^{i_1}} & \cdots & \alpha_{k-1}^{q^{i_{k-1}}}
\end{matrix} 
\right).
\end{equation*}

Besides $I=\{0,1,\cdots, k-1\}$, it is interesting to ask whether there exist other $I$ sharing the same property \eqref{eq:property_Moore}. Namely we would like to investigate the following research question.
\begin{question}\label{question:main}
	Determine the value of $q$, $n$ and $I$ such that $\det(M_{(\alpha_0,\ldots,\alpha_{k-1}),I})=0$ if and only if $\alpha_0,\ldots,\alpha_{k-1}$ are $\F_q$-linearly dependent for all $k$-tuples $(\alpha_0,\ldots,\alpha_{k-1})\in \mathbb{F}_{q^n}^k$.
\end{question}

For given $q$ and $n$, if $I$ is such that the condition in Question \ref{question:main} holds, then we say $I$ is a \emph{Moore exponent set} for $q$ and $n$.

Question \ref{question:main} is strongly related to \emph{maximum rank-distance codes} which are usually abbreviated to MRD codes. MRD codes have important applications in network coding and strong connections to semifield planes and linear sets in finite geometry; see \cite{sheekey_MRD_survey_arxiv} for a recent survey on them. It is already known that there are a huge number of inequivalent MRD codes consisting of $m\times n$ matrices over finite fields with $m-1<n$; see \cite{schmidt_number_MRD_2018}. However, there are only a few families of known MRD codes with $m=n$. In this case, every MRD code over $\F_q$ can be equivalently written as a set of $q$-polynomials. In particular, $I=\{i_0, i_1, i_2, \cdots, i_{k-1} \}$ is a Moore exponent set for $q$ and $n$ if and only if the set of $q$-polynomials
\[\cC=\left\{a_0 X^{q^{i_0}}+a_1 X^{q^{i_1}} + \cdots  a_{k-1} X^{q^{i_{k-1}}}: a_0,\cdots, a_{k-1}\in \F_{q^n} \right\}\]
defines an MRD code in $\F_q^{n\times n}$, i.e.\ each nonzero polynomial $f\in \cC$ has at most $q^k$ roots. The MRD code $\cC$ associated with $I$ has a special property: its right and left idealisers are both maximum; see \cite{liebhold_automorphism_2016,lunardon_kernels_2017} for details of the right (left) idealisers of MRD codes. For more details on this special type of MRD codes, see \cite{csajbok_MRD_maximum_idealisers_arxiv}. In \cite{napolitano_linear_arxiv}, the authors established a connection between the so-called  $h$-scattered linear sets and Moore exponent sets. We refer to \cite{csajbok_classes_2018,csajbok_newMRD_2017,csajbok_new_maximum_2018,lunardon_mrd-codes_2017,lunardon_generalized_2018,sheekey_new_arxiv,sheekey_new_2016,trombetti_family_2019} for more constructions of MRD codes and the links with finite geometries. 

It is easy to see that $I$ is a Moore exponent set if and only if $I+s=\{i+s: i \in I\}$ is so, whence we may always assume that the smallest element in $I$ is $0$. Besides $I=\{0,1,\cdots, k-1\}$, there are other known examples of Moore exponent sets.
\begin{itemize}
	\item $I=\{0,1,3\}$ and $\{0, 1, 2, 5\}$ for $n=7$ with odd $q$;
	\item $I=\{0,1,3\}$ and $\{0,1, 2, 3, 6\}$ for $n=8$ with $q\equiv 1\pmod{3}$;
	\item $I=\{0, d,\cdots, (k-1)d  \}$ for any $n$ satisfying $\gcd(d,n)=1$
\end{itemize}
The first two items have been discovered recently in \cite{csajbok_MRD_maximum_idealisers_arxiv}. The last one is equivalent to the so-called \emph{Delsarte-Gabidulin code} (sometimes also called a \emph{Generalized Gabidulin code}\cite{kshevetskiy_new_2005}).

It appears illusive to answer Question \ref{question:main} by giving a complete list of Moore exponent sets. Instead, we would like to present an asymptotic answer in this paper which also implies an asymptotic classification of MRD codes with maximum left and right idealisers. This result also partially answers Question 4.7 in \cite{csajbok_MRD_maximum_idealisers_arxiv}.
\begin{theorem}\label{th:main_Moore}
	Assume that $q>5$ and $I=\{0,i_1,i_2,\cdots, i_{k-1}\}$ with $0<i_1<\cdots<i_{k-1}$ is not an arithmetic progression. Then there exist an integer $N$ depending only on $I$ such that $I$ is not a Moore exponent set for $q$ and $n$ provided that $n>N$.
\end{theorem}

In fact, for $q\leq 5$, we can get the same result for almost each $I$ which is not an arithmetic progression. The precise conditions on $I$ and $q$ are presented in the following theorem, from which one can directly derive Theorem \ref{th:main_Moore}.  The main idea is to translate the determination of Moore exponent sets into an algebraic geometry problem. 
\begin{theorem}\label{th:main}
	Assume that $I=\{0,i_1,i_2,\cdots, i_{k-1}\}$ with $0<i_1<\cdots<i_{k-1}$ is not an arithmetic progression. Define $\mathcal{G}_k: G_k(X_1,\ldots,X_k)=0$ and $\mathcal{V}_I:  \frac{F_I(X_1,\ldots,X_k)}{G_k(X_1,\ldots,X_{k})}=0$, where 
	\begin{equation}\label{Eq:F_k}
	F_I(X_1,\ldots,X_k)=\det\left(
	\begin{matrix}
	X_1^{q^{i_0}} & X_1^{q^{i_1}} & \cdots & X_1^{q^{i_{k-1}}} \\ 
	X_2^{q^{i_0}} & X_2^{q^{i_1}} & \cdots &X_2^{q^{i_{k-1}}} \\ 
	\vdots&  \vdots & \ddots &  \vdots \\ 
	X_{k}^{q^{i_0}} & X_{k}^{q^{i_1}} & \cdots & X_{k}^{q^{i_{k-1}}}
	\end{matrix} 
	\right),
	\end{equation}
	and
	\begin{equation}\label{Eq:G_k}
	G_k(X_1,\ldots,X_k)=\det\left(
	\begin{matrix}
	X_1 & X_1^{q^{1}} & \cdots & X_1^{q^{{k-1}}} \\ 
	X_2 & X_2^{q^{1}} & \cdots &X_2^{q^{{k-1}}} \\ 
	\vdots&  \vdots & \ddots &  \vdots \\ 
	X_{k} & X_{k}^{q^{1}} & \cdots & X_{k}^{q^{{k-1}}}
	\end{matrix} 
	\right).
	\end{equation}
	Suppose that one of the following collections of conditions  is satisfied.
	\begin{enumerate}[label=(\alph*)]
		\item $i_2-i_0\neq 2(i_1-i_0)$;
		\item $i_2-i_0= 2(i_1-i_0)$, $k>3$ and $q\geq 7$;
		\item $i_2-i_0= 2(i_1-i_0)$, $k>3$, $q=3,4,5$ and $i_1-i_0>1$;
		\item $i_2-i_0= 2(i_1-i_0)$, $k>3$ and $q=2$ with $i_1-i_0>2$.
	\end{enumerate}
	There exists an integer $N$ such that $\mathcal{V}_I$ contains an $\F_{q^{n}}$-rational absolutely irreducible component and at least one $\mathbb{F}_{q^n}$-rational points not in $\mathcal{G}_k$ provided that $n>N$.
\end{theorem}
The exact value of $N$ in Theorem \ref{th:main} will be provided in Theorems \ref{th:curve_main} and \ref{th:general_main}.

The rest parts of this paper are organized as follows. In Section \ref{sec:pre} we introduce some tools and results from algebraic geometry; in Section \ref{sec:curve} we investigate the curve case of Theorem \ref{th:main}; finally in Section \ref{sec:general} we consider the general case of Theorem \ref{th:main} and present a complete proof.

\section{Preliminaries}\label{sec:pre}
To prove Theorem \ref{th:main}, we have to convert the original question into a problem of algebraic varieties over finite fields. In this section, we introduce some tools from algebraic geometry which will be used in the later parts.

An algebraic hypersurface is an algebraic variety that may be defined by a single implicit equation. An algebraic hypersurface defined over a field $\K$  is \emph{absolutely irreducible}  if the associated polynomial is irreducible over every algebraic extension of $\K$. An absolutely irreducible $\mathbb{K}$-rational component of a hypersurface $\mathcal{V}$, defined by the polynomial $F$, is simply an absolutely irreducible hypersurface such that the associated polynomial has coefficients in $\K$ and it is a factor of $F$.

%The next lemma is a standard result on non-absolutely irreducible varieties which can be found in .
\begin{lemma}\cite[Lemma~10]{hernando_proof_2011}
	\label{le:splitting_of_irreducible_polys}
	Let $F\in\F_q[X_1,\dots,X_m]$ be a polynomial of degree $d$, irreducible over~$\F_q$. Then there exists a natural number $s\mid d$ such that, over its splitting field, $F$ splits into $s$ absolutely irreducible polynomials, each of degree $d/s$.
\end{lemma}

\begin{lemma}\label{le:subvarieties}\cite[Lemma 2.1]{aubry_APN_2010}
Let $\mathcal{H}$ be a projective hypersurface and $\mathcal{X}$ a projective variety of dimension $n-1$ in $PG(n,q)$. If $\mathcal{X}\cap \mathcal{H}$ has a  non-repeated absolutely irreducible component defined over $\mathbb{F}_q$ then $\mathcal{X}$ has a  non-repeated absolutely irreducible component defined over $\mathbb{F}_q$.
\end{lemma}

Concerning the intersection number of two curves at a point, we need the following classical result which can be found in most of the textbooks on algebraic curves.
\begin{theorem}[B\'ezout's Theorem]\label{th:bezout}
	Let $\mathcal{A}$ and $\mathcal{B}$ be two projective plane curves over an algebraically closed field $\K$, having no component in common. Let $A$ and $B$ be the polynomials associated with $\mathcal{A}$ and $\mathcal{B}$ respectively. Then
	\[
	\sum_P I(P, \mathcal{A}\cap \mathcal{B})=(\deg A)(\deg B),
	\]
	where the sum runs over all points in the projective plane $\PG(2,\K)$.
\end{theorem}

We also need the following results to estimate the intersection number, which is not difficult to prove (see Janwa, McGuire, and Wilson~\cite[Proposition 2]{janwa_double-error-correcting_1995}).
\begin{lemma}
	\label{le:intersection_number_m_m1_coprime}
	Let $F$ be  a polynomial in $\F_q[X,Y]$ and suppose that $F=AB$. Let $P=(u,v)$ be a point in the affine plane $\mathrm{AG}(2,q)$ and write
	\[
	F(X+u,Y+v)=F_m(X,Y)+F_{m+1}(X,Y)+\cdots,
	\]
	where $F_i$ is zero or homogeneous of degree $i$ and $F_m\ne 0$. Let~$L$ be a linear polynomial and suppose that $F_m=L^m$ and $L\nmid F_{m+1}$. Then $I(P, \mathcal{A}\cap \mathcal{B})=0$, where $\mathcal{A}$ and $\mathcal{B}$ are the curves defined by $A$ and $B$ respectively.
\end{lemma}

The next result was proved in \cite[Lemma 4.3]{schmidt_planar_2014} for $q$ even case. Actually it still holds when $q$ is odd and  its proof is the same.
\begin{lemma}
	\label{le:intersection_number_linear_term}
	Let $F$ be a polynomial in $\F_q[X,Y]$ and suppose that $F=AB$. Let $P=(u,v)$ be a point in the affine plane $\mathrm{AG}(2,q)$ and write
	\[
	F(X+u,Y+v)=F_m(X,Y)+F_{m+1}(X,Y)+\cdots,
	\]
	where $F_i$ is zero or homogeneous of degree $i$ and $F_m\ne 0$. Let $L$ be a linear polynomial and suppose that $F_m=L^m$, $L \mid F_{m+1}$, $L^2 \nmid F_{m+1}$. Then $I(P, \mathcal{A}\cap \mathcal{B})=0$ or $m$, where $\mathcal{A}$ and $\mathcal{B}$ are the curves defined by $A$ and $B$ respectively.
\end{lemma}

The original Hasse-Weil bound is given in terms of genera of curves. Here we only need a weak version of it.
\begin{theorem}[Hasse-Weil Theorem]\label{th:HW-bound}
	For an absolutely irreducible curve $\cC$ in $\PG(2,q)$, then
	\[ \abs{ \# \cC(\F_q) - q-1 }\leq (d-1)(d-2)\sqrt{q},  \]
	where $d$ is the degree of the defining polynomial for $\cC$.
\end{theorem}

We also need two results concerning the number of rational points on an absolutely irreducible hypersurface.
\begin{theorem}\cite[Theorem 2]{zahid_nonsingular_2010}\label{th:zahid_0}
	Let $G$ be an absolutely irreducible hypersurface of degree $f$ defined over $\F_q$, and $H$ a hypersurface of degree $e$ defined over $\F_q$ not divisible by $G$. Then provided that 
	\[ q>\frac{1}{4}\left(  \alpha+ \sqrt{\alpha^2+4\beta} \right)^2 \]
	where $\alpha=(f-1)(f-2)$ and $\beta=5f^{13/3}+f(f+e-1)$, there is a nonsingular point of $G$ that is not a point of $H$.
\end{theorem}
\begin{theorem}\cite[Theorem 3]{zahid_nonsingular_2010}\label{th:zahid_1}
	Let $F$ be an absolutely irreducible hypersurface of degree $f$ defined over $\mathbb{F}_q$. Then provided that 
	$$q>\frac{3f^4-4f^3+5f^2}{2},$$
	there is a nonsingular point of $F$.
\end{theorem}
\begin{lemma}\label{lemmaGeneral}
Let $\mathcal{S}$ be a hypersurface containing $O=(0,0,\ldots,0)$ of the affine equation $F(X_1,\ldots,X_n)=0$, where

$$F(X_1,\ldots,X_n)=F_d(X_1,\ldots,X_n)+F_{d+1}(X_1,\ldots,X_n)+\cdots\textcolor{red}{,}$$
with $F_i$ the homogeneous part of degree $i$ of $F(X_1,\ldots,X_n)$ for $i=d,d+1,\cdots$. Let $P$ be an $\F_q$-rational simple point  of the variety 
 $$F_d(X_1,X_2,\ldots,X_{n-1},X_n)=0.$$ 
Then there exists an $\mathbb{F}_q$-rational plane $\pi$ through the line $\ell$ joining $O$ and $P$ such that $\pi\cap \mathcal{S}$ has $\ell$ as a non-repeated tangent $\mathbb{F}_q$-rational line at the origin and $\pi\cap \mathcal{S}$ has a non-repeated absolutely irreducible $\F_q$-rational component.
\end{lemma}
\proof
Without loss of generality we can suppose that $P=(0,0,\ldots,0,1)$. This means that $$F_d(X_1,\ldots,X_n)=X_n^{d-1}\left(\sum_{i=1}^{n-1}\alpha_i X_i\right)+\cdots,$$
with at least one of the $\alpha_i$'s different from $0$. Hence there exists at least one $(n-2)$-tuple $(\lambda_2, \lambda_3, \ldots, \lambda_{n-1})\in\mathbb{F}_q^{n-2}$ such that the line $m$ given by $\{(t,\lambda_2t, \ldots, \lambda_{n-1}t, 1) : t \in \overline{\F_q}\}$ intersects the variety $F_d(X_1,\ldots,X_{n-1},1)=0$ with multiplicity $1$ at $P$. This means that 
$$A=\alpha_1+\sum_{i=2}^{n-1}\alpha_i \lambda_i\neq 0.$$

Let $\pi$ be the plane generated by $m$ and $O$. Then $\pi$ is the set of points
$$\{(t,\lambda_2 t,\ldots,\lambda_{n-1}t,u) : t,u \in \overline{\mathbb{F}_q}\}.$$
The intersection between $\pi$ and $\mathcal{S}$ is given by 

$$F(X,\lambda_2 X,\ldots,\lambda_{n-1}X,Y)=F_d(X,\lambda_2 X,\ldots,\lambda_{n-1}X,Y)+\cdots=
AY^{d-1}X+\cdots$$

This shows that $X||F(X,\lambda_2 X,\ldots,\lambda_{n-1}X,Y)$, which means that the line $X=0$ in the plane $\pi$ is a non-repeated  tangent line at the origin for the $\mathbb{F}_q$-rational curve $\pi\cap \mathcal{S}$. Now consider the unique absolutely irreducible component $\mathcal{C}$ of $\pi\cap \mathcal{S}$ having $X=0$ as tangent at the origin. Since $X=0$ is also $\mathbb{F}_q$-rational, such a component must be $\mathbb{F}_q$-rational since it is fixed by the Frobenius morphism. Also, $\mathcal{C}$ cannot be repeated in $\pi\cap \mathcal{S}$, otherwise $X=0$ would be a repeated tangent line at the origin for $\pi\cap \mathcal{S}$.
\endproof

The next result can be simply proved by counting argument. It tells us the number of $\F_{q^{n}}$-rational points in $\cG_m$.
\begin{lemma}\label{le:counting}
	Let $m\leq n$ be two positive integers. The total number of points $(x_1,x_2, \cdots, x_m)\in \PG(m-1,q^n)$ such that $x_i$'s are linearly dependent equals $q^{n(m-1)} -  (q^n-q)(q^n-q^2)\cdots(q^n-q^{m-1})+\frac{q^{n(m-1)}-1}{q^n-1}$.
\end{lemma}

\section{Curves}\label{sec:curve}
Let $i,j$ be positive integers such that $j>i$ and consider $I=\{0,i,j\}$. Let $\cG_3$ and $\mathcal{V}_I$ be the curves of the affine equations $G_3(X,Y,T)=0$ and $F_I(X,Y,T)/G_3(X,Y,T)=0$, respectively, where  $F_I$ and $G_3$ are as in \eqref{Eq:F_k} and \eqref{Eq:G_k}. Note that $\cG_3$ coincides with the set of points in $PG(2,\overline{\F_q})$ lying on the union of all lines defined over $\F_q$.

\begin{theorem}\cite{borges_multi-Frobenius_2009}\label{th:Borges}
	Assume that $\gcd(i,j)=1$ and $j>2$. The curve $\mathcal{V}_I$ is absolutely irreducible 
%	Let $g_q$ denote its genus. Then
%	\[ g_q=(q^{j-i}+q^i)\left(\frac{q^j}{2}-(q^2+q+1)\right) +(q+1)(q^2+q+1)\]
	and the set of singular points of $\mathcal{V}_I$ is either $PG(2,q^{j-i})$ or $PG(2,q^{j-i})\setminus PG(2,q)$, in which the latter case happens if and only if $i=1$. Moreover
	\[\mathcal{V}_I \cap \cG_3 = 
	\begin{cases}
		(PG(2,q^{j-i})\setminus PG(2,q)) \cap \cG_3, & \text{if }i=1;\\
		PG(2,q^{j-i})\cap \cG_3, &\text{otherwise}.
	\end{cases}
	\]
%	and 
%	\[ \#\cC(\F_{q^k})=
%	\begin{cases}
%		(q^j-q^2)(q^j-q) + (q^2+q+1)(q^j-q), & \text{if } k=j,\\
%		(q^i-q^2)(q^i-q), & \text{if } k=i,\\
%		0, & \text{if }k=1.
%	\end{cases} \]
\end{theorem}

By Lemma \ref{le:counting}, we have
\[ \# (PG(2,q^{k})\cap \cG_3) = (q^k-q+1)\frac{q^3-1}{q-1}.\]
Hence
\[ \#(\mathcal{V}_I \cap \cG_3)= 
	\begin{cases}
		(q^{j-1}-q)\frac{q^3-1}{q-1}, & \text{if }i=1;\\
		(q^{j-i}-q+1)\frac{q^3-1}{q-1}, &\text{otherwise}.
		\end{cases}
\]
By the Hasse-Weil theorem (see Theorem \ref{th:HW-bound}), the number of $\F_{q^n}$-rational points of $\mathcal{V}_I$ satisfies
\begin{align}
\nonumber	\#\mathcal{V}_I(\F_{q^n}) &\geq q^n+1- (\ell-1)(\ell-2)\sqrt{q^n}\\
\label{eq:HW-bound}	&\geq q^n+1-q^{2j+n/2}-2q^{j+i+n/2}-q^{2i+n/2},
\end{align}
where $\ell=q^j+q^i-q^2-q$ is the degree of $\frac{F_I}{G_3}$.

When $\gcd(i,j)=1$, we can derive that 
\[ \#\mathcal{V}_I(\F_{q^n})> \#(\mathcal{V}_I \cap \cG_3)\]
provided $n>4j+2$. %\zhou{I have checked this bound quite carefully. The extreme case happens when $j-i=1$ and $q=2$}

When $\gcd(i,j)=d$ and $j\neq 2i$, $\mathcal{V}_I$ has two components
\[ \frac{F_I(X,Y,T)}{G_3(X,Y,T)}=\frac{F_I(X,Y,T)}{H_d(X,Y,T)} \cdot \frac{H_d(X,Y,T)}{G_3(X,Y,T)},\]
where 
$$H_d(X,Y,T)=\left|
\begin{array}{lll}
X&X^{q^d}&X^{q^{2d}}\\
Y&Y^{q^d}&Y^{q^{2d}}\\
T&T^{q^d}&T^{q^{2d}}\\
\end{array}
\right|.$$
Suppose that $i=i'd$ and $j=j'd$. Let $\cC'$ and $\mathcal{H}$ be the curves defined by $\frac{F_I(X,Y,T)}{H_d(X,Y,T)}=0$ and $H_d(X,Y,T)=0$, respectively. By Theorem \ref{th:Borges}, $\cC'$ is absolutely irreducible. Here we are just considering Theorem \ref{th:Borges} on $q^{\prime}=q^d$ with exponents $i^\prime$ and $j^{\prime}$.  It is obvious that $\cG_3$ is a component of $\mathcal{H}$. 

The degree of $\frac{F_I(X,Y,T)}{H_d(X,Y,T)}$ is $\ell'= q^j+q^i-q^{2d}-q^d$. By the Hasse-Weil bound, we have \
\begin{align*}
	\#\mathcal{C}'(\F_{q^n})&\geq q^n+1- (\ell'-1)(\ell'-2)\sqrt{q^n}\\
                           	 &\geq q^n+1-q^{2j+n/2}-2q^{j+i+n/2}-q^{2i+n/2},
\end{align*}
which is the same as the lower bound of $\#\mathcal{V}_I(\F_{q^n})$ obtained in \eqref{eq:HW-bound}.

Therefore, one of the following two conditions implies that $\#\mathcal{V}_I(\F_{q^n})\geq \#\mathcal{C}'(\F_{q^n})> \#(\mathcal{V}_I \cap \cG_3)$. 
\begin{itemize}
	\item $\#\cC'(\F_{q^n})> \#(\cC' \cap \mathcal{H})$ which holds if $n>4j+2$;
	\item $\cG_3(\F_{q^n})\subsetneq \mathcal{ H}(\F_{q^n})$ which holds if $\gcd(n,d)>1$.
\end{itemize}

Therefore we have proved the following result.

\begin{theorem}\label{th:curve_main}
	Let $i,j$ be two positive integer such that $j>i$ and $j\neq 2i$.
	For integer $n$ satisfying $n>4j+2$ or $\gcd(n,i,j)>1$ and any prime power $q$, $\{0, i, j\}$ is not a Moore exponent set.
\end{theorem}

\begin{remark}
	The lower bound on $n$ in Theorem \ref{th:curve_main} holds for all prime power $q$. When $q$ or the gap between $j$ and $i$ is large enough, one may get a slightly better lower bound $n>4j$.
\end{remark}

\section{General Case}\label{sec:general}
In this section, we investigate the general case of Theorem \ref{th:main} and we prove the following. 

\begin{theorem}\label{th:general_main}
	Suppose that $k>3$ and $I=\{0,i_1,i_2,\cdots, i_{k-1}\}$ with $0<i_1<\cdots<i_{k-1}$ is not an arithmetic progression. Assume that one of the following collections of conditions hold.
	\begin{enumerate}[label=(\alph*)]
		\item $i_2\neq 2i_1$;
		\item $i_2=2i_1$ and $q\geq 7$;
		\item $i_2=2i_1$, $q=3,4,5$ and $i_1>1$;
		\item $i_2=2i_1$, $q=2$ with $i_1>2$.
	\end{enumerate}
	For $n>\frac{13}{3}i_{k-1}+\log_q(13\cdot2^{10/3})$, $\mathcal{V}_I$ contains a simple $\mathbb{F}_{q^{n}}$-rational point which is not contained in $\cG_k$ (see \eqref{Eq:F_k} and \eqref{Eq:G_k}), whence $I$ is not a Moore exponent set.
\end{theorem}

Depending on whether $i_2=2i_1$, we separate the proof of the existence of an $\F_{q^{n}}$-rational absolutely irreducible component of $\mathcal{V}_I$ into two parts.

\begin{theorem}\label{th:case1}
	Let $I=\{0, i_1,\cdots, i_{k}\}$ be a set of positive integers satisfying $i_1<\cdots<i_{k}$. Let $F_I(X_1,\ldots,X_{k},1)$ and $G_{k+1}(X_1,\ldots,X_{k},1)$ be as in \eqref{Eq:F_k} and \eqref{Eq:G_k}.
	
	Suppose that $i_2\neq 2i_1$ and that $n>4i_{k-1}+2$. Then the affine hypersurface $\mathcal{V}_I$ of the affine equation $\frac{F_I(X_1,X_2,\ldots,X_{k},1)}{G_{k+1}(X_1,X_2,\ldots,X_{k},1)}=0$ contains a non-repeated $\mathbb{F}_{q^n}$-rational absolutely irreducible component. 
\end{theorem}

\proof
We prove the existence of an $\mathbb{F}_{q^n}$-rational absolutely irreducible component by induction on $k$. 

First let us consider the case $k=3$. Let $d_1=q^{i_2}+q^{i_1}+1$ and $d_2=q^2+q+1$. The homogeneous parts of the smallest degrees of $F_I(X_1,X_2,X_3,1)$ and $G_4(X_1,X_2,X_3,1)$ are 
\[\Phi_{d_1}(X_1,X_2,X_3)=\left|
\begin{array}{lll}
X_1&X_1^{q^{i_1}}&X_1^{q^{i_2}}\\
X_2&X_2^{q^{i_1}}&X_2^{q^{i_2}}\\
X_3&X_3^{q^{i_1}}&X_3^{q^{i_2}}\\
\end{array}
\right| \quad \text{and} \quad \Gamma_{d_2}(X_1,X_2,X_3)=\left|\begin{array}{lll}
X_1&X_1^{q}&X_1^{q^2}\\
X_2&X_2^{q}&X_2^{q^2}\\
X_3&X_3^{q}&X_3^{q^2}\\
\end{array}
\right|,\]
respectively.

Let $d_3=d_2-d_1$ and let $\Psi_{d_3}(X_1,X_2,X_3)$ be the homogeneous part of the smallest degree of the polynomial $H(X_1,X_2,X_3)=\frac{F_I(X_1,X_2,X_3,1)}{G_4(X_1,X_2,X_3,1)}$. Then $\Phi_{d_1}(X_1,X_2,X_3)=\Gamma_{d_2}(X_1,X_2,X_3)\Psi_{d_3}(X_1,X_2,X_3)$, which means that the tangent cone at $O=(0,0,0)$ of $\mathcal{V}_I$ is given by 
$$\Psi_{d_3}(X_1,X_2,X_3)=\frac{\left|
	\begin{array}{lll}
	X_1&X_1^{q^{i_1}}&X_1^{q^{i_2}}\\
	X_2&X_2^{q^{i_1}}&X_2^{q^{i_2}}\\
	X_3&X_3^{q^{i_1}}&X_3^{q^{i_2}}\\
	\end{array}
	\right|}
{\left|\begin{array}{lll}
	X_1&X_1^{q}&X_1^{q^2}\\
	X_2&X_2^{q}&X_2^{q^2}\\
	X_3&X_3^{q}&X_3^{q^2}\\
	\end{array}
	\right|}=\frac{\left|
	\begin{array}{lll}
	X_1&X_1^{q^{i_1}}&X_1^{q^{i_2}}\\
	X_2&X_2^{q^{i_1}}&X_2^{q^{i_2}}\\
	X_3&X_3^{q^{i_1}}&X_3^{q^{i_2}}\\
	\end{array}
	\right|}
{\left|\begin{array}{lll}
	X_1&X_1^{q^d}&X_1^{q^{2d}}\\
	X_2&X_2^{q^d}&X_2^{q^{2d}}\\
	X_3&X_3^{q^d}&X_3^{q^{2d}}\\
	\end{array}
	\right|}\cdot
\frac{\left|
	\begin{array}{lll}
	X_1&X_1^{q^d}&X_1^{q^{2d}}\\
	X_2&X_2^{q^d}&X_2^{q^{2d}}\\
	X_3&X_3^{q^d}&X_3^{q^{2d}}\\
	\end{array}
	\right|}
{\left|\begin{array}{lll}
	X_1&X_1^{q}&X_1^{q^2}\\
	X_2&X_2^{q}&X_2^{q^2}\\
	X_3&X_3^{q}&X_3^{q^2}\\
	\end{array}
	\right|},$$
where $d=\gcd(i_1,i_2)$ and we denote the first component by $C(X_1,X_2,X_3)$.

By Theorem \ref{th:Borges}, the curve defined by $C(X_1,X_2,X_3)=0$ is absolutely irreducible and the set of its singular points is either $PG(2,q^{i_2-i_1})$ or $PG(2,q^{i_2-i_1})\setminus PG(2,q^d)$, in which the latter case happens if and only if $i_1=d$. By the Hasse-Weil theorem, the number of its $\F_{q^n}$-rational simple points is at least 
$$q^n+1- (\ell-1)(\ell-2)\sqrt{q^n}-q^{2(i_2-i_1)}-q^{i_2-i_1}-1,$$
where $\ell=q^{i_2}+q^{i_1}-q^{2d}-q^d$.

When $n> 4i_2+2$, which always holds provided that $n>4i_{k-1}+2$, it is straightforward to check that there is at least one simple point on the curve defined by $C(X_1,X_2,X_3)=0$.

Hence, by the existence of a simple $\F_{q^n}$-rational point $P$ in the curve $C(X_1,X_2,X_3)$, one can show, by Lemma \ref{lemmaGeneral}, that there exists an $\mathbb{F}_{q^n}$-rational plane $\pi$ through the origin and $P$ such that $\pi\cap \mathcal{V}_I$ has a non-repeated absolutely irreducible component of $\pi\cap \mathcal{V}_I$. By Lemma \ref{le:subvarieties} there is a  non-repeated $\mathbb{F}_{q^n}$-rational component $\mathcal{W}_I$ in $\mathcal{V}_I$ which is absolutely irreducible.

Suppose now that for each $I=\{0, i_1,\cdots, i_{j}\}$, $j\leq s-1$, the hypersurface $\mathcal{V}_I$ contains a  non-repeated $\mathbb{F}_{q^n}$-rational absolutely irreducible component. Let us prove the case $I=\{0, i_1,\cdots, i_{s}\}$. The tangent cone at the origin of $\mathcal{V}_I$ is the projective closure of the affine variety $\mathcal{V}_{I^\prime}$, where $I^{\prime}=\{0, i_1,\cdots, i_{s-1}\}$, which by induction contains a non-repeated absolutely irreducible $\mathbb{F}_{q^n}$-rational component $\mathcal{W}_{I^\prime}$.

Note that the degree of $\mathcal{V}_{I^{\prime}}$ is $d=q^{i_{s-1}}+q^{i_{s-2}}+\cdots q^{i_{1}}-(q^{s-1}+q^{s-2}+\cdots q)$. By assumption $n>4i_{k-1}+2$, which implies $q^n> \frac{3}{2}d^4> \frac{3d^4-4d^3+5d^2}{2}$. By Theorem \ref{th:zahid_1}, the component $\mathcal{W}_{I^\prime}$, which has degree at most $d$, has a simple $\mathbb{F}_{q^n}$-rational point. This shows, by Lemma \ref{lemmaGeneral}, that there exists an $\mathbb{F}_{q^n}$-rational plane $\pi$ through the origin and $P$ such that $\pi\cap \mathcal{V}_I$ contains a  non-repeated absolutely irreducible component defined over $\mathbb{F}_{q^n}$. 

Take an arbitrary $3$-dimensional subspace $S$ containing $\pi$. The variety $S\cap \mathcal{V}_I$ has $\pi\cap \mathcal{V}_I$ as a hyperplane section. By Lemma \ref{le:subvarieties}, $S\cap \mathcal{V}_I$ possesses a  non-repeated absolutely irreducible $\mathbb{F}_{q^n}$-rational component. Using this inductive argument, it is readily seen that there is a  non-repeated absolutely irreducible $\mathbb{F}_{q^n}$-rational component in $\mathcal{V}_I$.
\endproof

Next let us turn to the case $i_2=2i_1$. We define 
\begin{equation}\label{Eq:L}
L(U,Z_1,\ldots,Z_{k-3})=\left|
\begin{array}{llll}
U&U^{q^{i_2}}&\cdots&U^{q^{i_{k-1}}}\\
Z_1&Z_1^{q^{i_2}}&\cdots&Z_1^{q^{i_{k-1}}}\\
\vdots&\vdots&&\vdots\\
Z_{k-3}&Z_{k-3}^{q^{i_2}}&\cdots&Z_{k-3}^{q^{i_{k-1}}}\\
1&1&\cdots&1\\
\end{array}
\right|,
\end{equation}
\begin{equation}\label{Eq:M}
M(U,Z_1,\ldots,Z_{k-3})=\left|
\begin{array}{llll}
U&U^{q^{i_2-i_1}}&\cdots&U^{q^{i_{k-1}-i_1}}\\
Z_1&Z_1^{q^{i_2-i_1}}&\cdots&Z_1^{q^{i_{k-1}-i_1}}\\
\vdots&\vdots&&\vdots\\
Z_{k-3}&Z_{k-3}^{q^{i_2-i_1}}&\cdots&Z_{k-3}^{q^{i_{k-1}-i_1}}\\
1&1&\cdots&1\\
\end{array}
\right|,
\end{equation}

\begin{equation}\label{Eq:N}
N(Z_1,\ldots,Z_{k-3})=\left|
\begin{array}{lll}
Z_1^{q^{i_2-i_1}}&\cdots&Z_1^{q^{i_{k-1}-i_1}}\\
\vdots&&\vdots\\
Z_{k-3}^{q^{i_2-i_1}}&\cdots&Z_{k-3}^{q^{i_{k-1}-i_1}}\\
1&\cdots&1\\
\end{array}\right|,
\end{equation}

\begin{equation}\label{Eq:R}
R(Z_2,\ldots,Z_{k-3})=\left|
\begin{array}{lll}
Z_2^{q^{i_3-i_1}}&\cdots&Z_2^{q^{i_{k-1}-i_1}}\\
\vdots&&\vdots\\
Z_{k-3}^{q^{i_3-i_1}}&\cdots&Z_{k-3}^{q^{i_{k-1}-i_1}}\\
1&\cdots&1\\
\end{array}
\right|.
\end{equation}

\begin{proposition}\label{prop:i_2=2i_1}
	Suppose $k>3$ and $i_2=2i_1$. If $M(U,Z_1,\ldots,Z_{k-3})$  divides $L(U,Z_1,\ldots,Z_{k-3})$ then $i_1\mid i_j$ for each $j=1,\ldots,k-1$.
\end{proposition}
\proof
Suppose, by way of contraction, that not all the $i_j$ are divisible by $i_1$.

Clearly, if $M(U,Z_1,\ldots,Z_{k-3})$ divides $L(U,Z_1,\ldots,Z_{k-3})$,  then in particular $M(U,V,z_2,\ldots,z_{k-3})$ also divides $L(U,V,z_2,\ldots,z_{k-3})$ for all $(z_2,\ldots,z_{k-3})$. Choose $(z_2,\ldots,z_{k-3})\in \overline{\mathbb{F}_q}^{k-3}$ such that none of the lower $(k-3)\times (k-3)$\ minors in the determinant  $M(U,V,z_2,\ldots,z_{k-3})$ vanishes. 

The tangent cone at the origin of the curve $\mathcal{D}$ defined by the affine equation $M(U,V,z_2,\ldots,z_{k-3})=0$ is $R(z_2,\ldots,z_{k-3}) (UV^{q^{i_1}}-VU^{q^{i_1}})$. Now the origin is an ordinary $(q^{i_1}+1)$-fold singular point of $\mathcal{D}$ (since the polynomial $UV^{q^{i_1}}-VU^{q^{i_1}}$ factorizes in non-repeated linear factors over ${\mathbb{F}}_{q^{i_1}}$). Therefore there are exactly $q^{i_1}+1$ branches centered at the origin and they correspond to the elements of $\mathbb{F}_{q^{i_1}}\cup \{\infty\}$.

Given $\lambda\in \mathbb{F}_{q^{i_1}}\cup \{\infty\}$, let  $\gamma_{\lambda}$ denote the corresponding branch of the curve $\mathcal D$ centered at the origin. Since $i_1$ does not divide all the $i_j$ the line $V=\lambda U$ is not a component of  $M(U,V, z_2, \cdots, z_{k-3})=0$ and  the branch $\gamma_{\lambda}$ is of  the type $(t,\lambda t+\mu t^{\alpha}+\cdots)$ for some nonzero $\mu\in \overline{\mathbb{F}_q}$ and $\alpha>1$.  So $\gamma_{\lambda}$ belongs to the curve $L(U,V,z_2,\ldots,z_{k-3})=0$ too. Recall that 
\begin{align*}
&L(t,\lambda t +\mu t^{\alpha}+\cdots,z_{2},\ldots,z_{k-3})\\
=&\left|
\begin{array}{llll}
t&t^{q^{2i_1}}&\cdots&t^{q^{i_{k-1}}}\\
\lambda t +\mu t^{\alpha}+\cdots&\lambda t^{q^{2i_1}} +\mu^{q^{2i_1}} t^{\alpha{q^{2i_1}}}+\cdots&\cdots&\lambda^{q^{i_{k-1}}} t^{q^{i_{k-1}}} +\mu^{q^{i_{k-1}}} t^{\alpha{q^{i_{k-1}}}}+\cdots\\
\vdots&\vdots&&\vdots\\
z_{k-3}&z_{k-3}^{q^{2i_1}}&\cdots&z_{k-3}^{q^{i_{k-1}}}\\
1&1&\cdots&1\\
\end{array}
\right|,
\end{align*}
we must have that the above power series in $t$ vanishes. We may subtract the second row by the first row times $\lambda$. By checking the term of the smallest degree, we must have $\alpha+q^{2i_1}=\alpha q^{2i_1}+1$ which yields $\alpha=1$ or $q^{2i_1}=1$, a contradiction.
\endproof

\begin{theorem}\label{th:case2}
Let $k$ be an integer larger than $3$ and
$$F(X,Y)=
\left|
\begin{array}{lllll}
X&X^{q^{i_1}}&X^{q^{i_2}}&\cdots&X^{q^{i_{k-1}}}\\
Y&Y^{q^{i_1}}&Y^{q^{i_2}}&\cdots&Y^{q^{i_{k-1}}}\\
z_1&z_1^{q^{i_1}}&z_1^{q^{i_2}}&\cdots&z_1^{q^{i_{k-1}}}\\
\vdots&\vdots&\vdots&&\vdots\\
z_{k-3}&z_{k-3}^{q^{i_1}}&z_{k-3}^{q^{i_2}}&\cdots&z_{k-3}^{q^{i_{k-1}}}\\
1&1&1&\cdots&1\\
\end{array}
\right|,$$
and
$$G(X,Y)=
\left|
\begin{array}{lllll}
X&X^q&X^{q^{2}}&\cdots&X^{q^{k-1}}\\
Y&Y^q&Y^{q^{2}}&\cdots&Y^{q^{k-1}}\\
z_1&z_1^q&z_1^{q^{2}}&\cdots&z_1^{q^{k-1}}\\
\vdots&\vdots&\vdots&&\vdots\\
z_{k-3}&z_{k-3}^q&z_{k-3}^{q^{2}}&\cdots&z_{k-3}^{q^{k-1}}\\
1&1&1&\cdots&1\\
\end{array}
\right|,$$
Assume that there exists $(z_1,\ldots,z_{k-3})\in \mathbb{F}_{q^n}^{k-3}$ such that $N(z_1,\ldots,z_{k-3})\neq0$  and  $M(U,z_1,\ldots,z_{k-3})$  does not divide $L(U,z_1,\ldots,z_{k-3})$. If one of the following conditions is satisfied,
\begin{itemize}
	\item $q\geq 7$,
	\item $q=3,4,5$ and $i_1>1$,
	\item $q=2$ and $i_1>2$,
\end{itemize}
then the curve $\mathcal{C}$ of the affine equation $\frac{F(X,Y)}{G(X,Y)}=0$ contains a  non-repeated $\mathbb{F}_{q^n}$-rational absolutely irreducible component.  
\end{theorem}
\proof
We want to study the intersection multiplicity of two putative components $\mathcal{A}$ and $\mathcal{B}$ of $\mathcal{C}$ at its singular points.

By direct computation, affine singular points of $F(X,Y)=0$ satisfy

\begin{eqnarray*}
\frac{\partial F(X,Y)}{\partial X}&=& (M(Y,z_1,\ldots,z_{k-3}))^{q^{i_1}},\\  
\frac{\partial F(X,Y)}{\partial Y}&=& (M(X,z_1,\ldots,z_{k-3}))^{q^{i_1}},\\
F(X,Y)&=&0.\\
\end{eqnarray*}

Consider now a singular point $(\alpha,\beta)$ of $\mathcal{C}$. Then 
\[M(\alpha,z_1,\ldots,z_{k-3})=M(\beta,z_1,\ldots,z_{k-3})=0.\] 
Expanding $F(X+\alpha,Y+\beta)$, one can see that the terms of the smallest degree appearing in it are
$$L(\alpha,z_1,\ldots,z_{k-3})Y^{q^{i_1}}-L(\beta,z_1,\ldots,z_{k-3})X^{q^{i_1}}+(N(z_1,\ldots,z_{k-3}))^{q^{i_1}}(XY^{q^{i_1}}-YX^{q^{i_1}})+\cdots,$$
where $L$ is as in \eqref{Eq:L}.

\begin{itemize}
	\item If one between $\alpha$ and $\beta$ does not satisfy $L(U,z_1,\ldots,z_{k-3})=0$ then $P$ has multiplicity $q^{i_1}$. Since \[XY^{q^{i_1}}-YX^{q^{i_1}}=XY\prod_{\omega\in \F_{q^{i_1}}^*}(Y-\omega X)\] 
	and \[L(\alpha,z_1,\ldots,z_{k-3})Y^{q^{i_1}}-L(\beta,z_1,\ldots,z_{k-3})X^{q^{i_1}}=(c_\alpha Y-c_\beta X)^{q^{i_1}}\] 
	with $c_\alpha^{q^{i_1}}=L(\alpha,z_1,\ldots,z_{k-3})$ and $c_\beta^{q^{i_1}}=L(\beta,z_1,\ldots,z_{k-3})$, the intersection multiplicity of the putative components $\mathcal{A}$ and $\mathcal{B}$ at $P$ is either $0$ or $q^{i_1}$, by Lemmas  \ref{le:intersection_number_m_m1_coprime} and \ref{le:intersection_number_linear_term}.
    \item If $L(\alpha,z_1,\ldots,z_{k-3})=L(\beta,z_1,\ldots,z_{k-3})=0$ then the intersection multiplicity of the two components at $P$ is at most $(q^{i_1}+1)^2/4$. Since by assumption $M(U,z_1,\ldots,z_{k-3})$  does not divide $L(U,z_1,\ldots,z_{k-3})$, the number of points of the second type is  at most $q^{2(\deg(M)-1)}=q^{2(i_{k-1}-i_1-1)}$.
    \item Consider now an ideal point $P=(\alpha,\beta,0)$. Such a point is equivalent (up to a change of variables) to an affine singular point  of the the curve $\mathcal{C}$. So we can suppose that the intersection multiplicity of the two components $\mathcal{A}$ and $\mathcal{B}$ at $P$ is at most $(q^{i_1}+1)^2/4$. The total number of ideal points is at most $(q^{i_{k-1}-i_{k-2}}+1)$, since the term of the highest degree in $F(X,Y)$ is  $\left(XY^{q^{i_{k-1}-i_{k-2}}}-X^{q^{i_{k-1}-i_{k-2}}}Y\right)^{q^{i_{k-2}}}$.
\end{itemize}

The largest possible value for the sum of the multiplicities of intersection of two components $\mathcal{A}$ and $\mathcal{B}$ of $\mathcal{C}$ is 

\begin{eqnarray}\label{Eq:MultiplicitiesGen}
\tau&=&\underbrace{\left(q^{2(i_{k-1}-i_1)}-q^{2(i_{k-1}-i_1-1)}\right)}_{\text{Affine  Type I}}\cdot q^{i_1}+\underbrace{q^{2(i_{k-1}-i_1-1)}}_{\text{Affine Type II}}\cdot (q^{i_1}+1)^2/4\nonumber\\
&&+\underbrace{(q^{i_{k-1}-i_{k-2}}+1)}_{\text{Infinity}}\cdot (q^{i_1}+1)^2/4.
\end{eqnarray}
From \eqref{Eq:MultiplicitiesGen}, we can derive an upper bound on $\tau$, denoted by $B_\tau$,
\begin{eqnarray}\label{Eq:upperbound_tau}
	\tau&\leq& q^{2i_{k-1}} \left(\frac{1}{q^{i_1}}\left( 1-\frac{1}{q^2} \right) + \frac{(q^{i_1}+1)^2}{4q^{2i_1+2}} + \frac{q^{i_1}+4}{4q^{i_{k-1}+i_{k-2}-i_1}}\right)+\frac{(q^{i_1}+1)^2}{4}\nonumber\\
	& \leq& q^{2i_{k-1}} \left(\frac{1}{q^{i_1}}\left( 1-\frac{1}{q^2} \right) + \frac{(q^{i_1}+1)^2}{4q^{2i_1+2}} + \frac{1}{4q^{2k-5}} +\frac{1}{q^{2k-4}} \right)+\frac{(q^{i_1}+1)^2}{4},
\end{eqnarray}
because
\begin{eqnarray*}
	    \frac{ (q^{i_{k-1}-i_{k-2}}+1)(q^{i_1}+1)^2}{4}&=&\frac{q^{i_{k-1}-i_{k-2}+i_1}}{4}\left( q^{i_1} + 2+ \frac{1}{q^{i_1}}\right) + \frac{(q^{i_1}+1)^2}{4}\\
	&\leq& \frac{(q^{i_1}+4)q^{2i_{k-1}}}{4q^{i_{k-1}+i_{k-2}-i_1}} + \frac{(q^{i_1}+1)^2}{4}.
\end{eqnarray*}

Assume that
\[F(X,Y)/G(X,Y)=W_1(X,Y)W_{2}(X,Y)\cdots W_{r}(X,Y)\]
is the decomposition over $\F_{q^n}$ with $\deg W_\ell=d_\ell$. It is clear that the degree of $F(X,Y)/G(X,Y)$ is
$$d=\sum_{k=1}^{r} d_\ell=q^{i_{k-1}}+q^{i_{k-2}}-q^{k-1}-q^{k-2}.$$ 
Since we have already shown that for any two components $\mathcal{A}$ and $\mathcal{B}$ their total intersection number has \eqref{Eq:MultiplicitiesGen} as upper bound,  $W_{\ell_1}$ and $W_{\ell_2}$ must be relatively prime for any distinct $\ell_1$ and $\ell_2$.

By Lemma \ref{le:splitting_of_irreducible_polys}, there exist natural numbers $s_\ell$ such that $W_{\ell}$ splits into $s_\ell$ absolutely irreducible factors over $\overline{\F_{q^n}}$, each of degree $d_\ell/s_\ell$. Assume, by way of contradiction, that $\mathcal{C}$ has no absolutely irreducible component over $\F_{q^n}$, i.e.\ $s_\ell>1$ for $\ell=1,2,\dots, r$. Define two polynomials $A(X,Y,T)$ and $B(X,Y,T)$ by
\[A(X,Y,T)=\prod_{\ell=1}^r\prod _{j=1}^{\lfloor s_\ell/2 \rfloor }  Z_{\ell,j} (X,Y,T), \qquad B(X,Y,T)=\prod_{\ell=1}^{r}\prod _{j=\lfloor s_\ell/2 \rfloor+1}^{s_\ell}  Z_{\ell,j} (X,Y,T),\]
where $Z_{\ell,1}(X,Y,T), \ldots, Z_{\ell,s_\ell}(X,Y,T)$ are the absolutely irreducible components of $W_{\ell}(X,Y,T)$. Let $\alpha$ and $\alpha+\beta$ be the degrees of $A(X,Y,T)$ and $B(X,Y,T)$ respectively. Then 
$$2\alpha+\beta=d, \qquad \beta\leq \alpha, \qquad \beta \leq \frac{d}{3}.$$
Let $\mathcal{A}$ and $\mathcal{B}$ be the curves defined by $A(X,Y,T)$ and $B(X,Y,T)$, respectively. It is clear that
\begin{equation*}
(\deg A)( \deg B)=(\alpha+\beta)\alpha = \frac{d^2-\beta^2}{4} \ge \frac{2}{9}d^2.
\end{equation*}
By looking at the value of $d$, we have
{\small
\begin{equation}\label{Eq:degree}
		 \frac{2}{9}d^2=\frac{2}{9} \left(\left(q^{{i_{k-1}}}-q^{k-1}\right)^2+2(q^{{i_{k-1}}}-q^{k-1})(q^{i_{k-2}}-q^{k-2})+(q^{i_{k-2}}-q^{k-2})^2  \right).
\end{equation}}

By comparing the upper bound $B_\tau$ for $\tau$ in \eqref{Eq:upperbound_tau} and the value of $\frac{2}{9}d^2$ in \eqref{Eq:degree}, we see that if $\frac{2}{9}d^2-B_\tau>0$,
then $\tau<(\deg A)(\deg B)$ which is  a contradiction to B\'ezout's theorem; see Theorem \ref{th:bezout}. As $i_{k-1}\geq i_1+k-1$ and $i_{k-2}\geq i_1+k-2$, we have $q^{i_{k-1}}\geq q^{i_1+k-1}$, $q^{i_{k-2}}\geq q^{i_1+k-2}$ and
\begin{align*}
	&\frac{2}{9}d^2-B_\tau\\
	\geq &q^{{2(i_{1}+k-1)}}\left( \frac{2}{9}\left(1-\frac{1}{q^{i_1}}\right)^2-  \left(\frac{1}{q^{i_1}}\left( 1-\frac{1}{q^2} \right) + \frac{(q^{i_1}+1)^2}{4q^{2i_1+2}} + \frac{1}{4q^{2k-5}} +\frac{1}{q^{2k-4}} \right)\right)\\
	&+\frac{2}{9} (q^{i_1}-1)^2 q^{2(k-2)}(2q+1)-\frac{(q^{i_1}+1)^2}{4}.
\end{align*}

One may make a computer program to check that one of the assumptions on $q$ and $i_1$ implies $\frac{2}{9}d^2-B_\tau>0$. Therefore, there is at least one non-repeated absolutely irreducible component defined by $W_\ell$ for some $\ell\in \{1,2,\cdots, r\}$.
\endproof

Suppose that $I=\{0,i_1,i_2,\cdots, i_{k-1}\}$ with $i_2=2i_1$ is not an arithmetic progression. Then, by the proof of Proposition \ref{prop:i_2=2i_1}, there does exist $(z_1,\ldots,z_{k-3})\in \mathbb{F}_{q^n}^{k-3}$ such that $N(z_1,\ldots,z_{k-3})\neq0$  and  $M(U,z_1,\ldots,z_{k-3})$  does not divide $L(U,z_1,\ldots,z_{k-3})$. Thus, by Theorem \ref{th:case2}, there exists a non-repeated $\F_{q^n}$-rational absolutely irreducible component in $\cC$. Proceeding as in the proof of Theorem \ref{th:case1}, we can consider subspaces of increasing dimension containing the plane curve $\cC$. Therefore, by using Lemma \ref{le:subvarieties} recursively, we obtain the following result.
\begin{corollary}\label{coro:case2}
	Suppose that $i_2=2i_1$ and $I=\{0,i_1,i_2,\cdots, i_{k-1}\}$ is not an arithmetic progression. If one of the following conditions satisfies,
	\begin{itemize}
		\item $q\geq 7$,
		\item $q=3,4,5$ and $i_1>1$,
		\item $q=2$ and $i_1>2$,
	\end{itemize}
	then the variety $\mathcal{V}_I$ contains a  non-repeated $\mathbb{F}_{q^n}$-rational absolutely irreducible component. 
\end{corollary}

Finally we can prove our main result in this section.
\begin{proof}[Proof of Theorem \ref{th:general_main}]
	By Theorem \ref{th:case1} and Corollary \ref{coro:case2}, there exists an $\F_{q^n}$-rational absolutely irreducible component $\mathcal{V}'_I$ of $\mathcal{V}_I$. 
	The polynomial $G_I(X_1,\ldots,X_{k},1)$ is the product of all the $(q^k-1)/(q-1)$ non-proportional $\mathbb{F}_q$-linear forms of degree $1$ in $X_1,X_2,\ldots,X_{k}$ (repeated only once). The polynomial $F_I(X_1,X_2,\ldots,X_{k},1)$ is a separable linearized polynomial in each of the variables $X_1,\ldots,X_k$ and each (linear) factor of $G_I(X_1,X_2,\ldots,X_{k},1)$ cannot be repeated in $F_I(X_1,X_2,\ldots,X_{k},1)$. This shows that $F_I(X_1,X_2,\ldots,X_{k},1)/G_I(X_1,X_2,\ldots,X_{k},1)$  and $G_I(X_1,X_2,\ldots,X_{k},1)$ have no factors in common. In particular, $\mathcal{V}'_I$ is not a component of $\cG_k$.
	
	Let us consider $\mathcal{V}'_I$ and $\cG_k$ using Theorem \ref{th:zahid_0}. It is clear that the degree of $\cG_k$ is $e=\sum_{j=0}^{k-1}q^j$ and the degree  $\mathcal{V}'_I$ is $f\leq \sum_{j=0}^{k-1}q^{i_j}-e$. By computation,
	\begin{eqnarray*}
		&    &\alpha+ \sqrt{\alpha^2+4\beta}\\
		&=& f^2-3f+2+\sqrt{ f^4-6f^3+13f^2-12f+4 + 20f^{13/3}+4f^2+ 4fe-4f}\\
		&=& f^2-3f+2+\sqrt{ 20f^{13/3}+f^4-6f^3+17f^2-16f+ 4fe+4}.
	\end{eqnarray*}
	Let $\delta_q = \sum_{j=0}^{k-1}q^{-j}$. Then $e=\delta_qq^{k-1}$ and $f<\delta_q q^{i_{k-1}}$. Plug them into the equation above, together with the fact that $i_{k-1}\geq k> 3$, one gets
	\[\alpha+ \sqrt{\alpha^2+4\beta}< \sqrt{26}(\delta_q q^{i_{k-1}})^{13/6}.\]

	When integer $n>\frac{13}{3}i_{k-1}+\log_q(13\cdot2^{10/3})$, we have
	\[q^n>\frac{13}{2}(2q^{i_{k-1}})^{13/3} >\frac{13}{2}(\delta_q q^{i_{k-1}})^{13/3} >\frac{1}{4}\left(  \alpha+ \sqrt{\alpha^2+4\beta} \right)^2.\] 
	By Theorem \ref{th:zahid_0}, there exists a nonsingular $\mathbb{F}_{q^n}$-rational point of $\mathcal{V}'_I$ that is not a point of $\cG_k$.
\end{proof}

\begin{remark}
	In this paper, we have obtained a complete asymptotic classification result for Moore exponent sets for $q>5$ and a partial result for $q=2,3,4,5$. It appears that a complete answer could be also obtained for $q=2,3,4$ and $5$. To prove this result, one possible way is to get a better estimation of the upper bound for the sum of the multiplicities of intersection of two putative components in Theorem \ref{th:case2}.
\end{remark}

\section*{acknowledgment}
The work of Daniele Bartoli was supported by the Italian National Group for Algebraic and Geometric Structures and their Applications (GNSAGA - INdAM). Yue Zhou is partially supported by National Natural Science Foundation of China (No.\ 11771451) and Natural Science Foundation of Hunan Province (No.\ 2019RS2031). This work is partially done during the visit of the second author to the University of Perugia. He would like to thank Daniele Bartoli and Massimo Giulietti for their hospitality.

%\bibliographystyle{abbrv}
%\bibliography{C:/Documents/References/Reference_math}

\end{document}